\documentclass[10pt]{amsart}
\usepackage{natbib}
\usepackage{epsfig,psfrag,amsmath,amssymb,longtable,dcolumn}
\usepackage[ruled,vlined,norelsize]{algorithm2e}

\def\co{\colon\thinspace}

\newcommand{\sig}{\sigma}

\newcommand{\strdiv}{\parallel}
\newcommand{\divides}{\mid}

\newcommand{\FF}{$F_5$ }

\DeclareMathOperator{\Mon}{Mon}
\DeclareMathOperator{\Rad}{Rad}
\DeclareMathOperator{\LM}{lm}
\DeclareMathOperator{\LT}{lt}
\DeclareMathOperator{\LC}{lc}
\DeclareMathOperator{\tail}{tail}
\DeclareMathOperator{\stdmon}{stdMon}
\DeclareMathOperator{\NF}{NF}
\DeclareMathOperator{\interred}{interred}

\DeclareMathOperator{\poly}{poly}

\newtheorem{rk}{Remark}
\newtheorem{thm}{Theorem}
\newtheorem{lem}{Lemma}
\newtheorem{cor}{Corollary}

\theoremstyle{definition}

\newtheorem{defn}{Definition}

\theoremstyle{remark}

\begin{document}

\title[A non-commutative $F_5$ algorithm and Loewy layers]{A non-commutative $F_5$ algorithm with an application to the
  computation of Loewy layers}
\author{Simon A. King}
\address{Department of Mathematics and Computer Science\\Institute of Mathematics\\ Friedrich Schiller University Jena}
\email{simon.king@uni-jena.de}
\thanks{DFG project GR 1585/6--1}
\begin{abstract}
\noindent
We provide a non-commutative version of the $F_5$ algorithm, namely for
right-modules over path algebra quotients. It terminates, if the path algebra
quotient is a basic algebra. In addition, we use the $F_5$
algorithm in negative degree monomial orderings to compute Loewy layers.
\end{abstract}
\maketitle

\tableofcontents

\section{Introduction}

In this paper, we present a non-commutative version of the \FF algorithm, in
the setting of finitely generated sub-modules of free right-modules of finite
rank over path algebra quotients. It terminates, if the path algebra quotient
is a basic algebra. The \FF algorithm is usually thought of as an algorithm
for the computation of Gr\"obner bases. It improves efficiency of the
computation by discarding many ``useless'' critical pairs, namely critical
pairs whose S-polynomials would reduce to zero, in Buchberger's algorithm.  In
addition, we show that the \FF algorithm immediately yields the Loewy layers
of a right-module over a basic algebra, provided that a negative degree
monomial ordering is used.

A \emph{basic algebra} is a finite dimensional quotient of a path algebra, by
relations that are at least quadratic. Basic algebras are useful in the study
of modular group algebras of finite groups: If $p$ is a prime dividing the
order of a finite group $G$, then there is some finite extension field $K$
over $\mathbb F_p$, so that $KG$ is Morita equivalent to a basic algebra
$\mathcal A$ over some finite quiver, whose connected components correspond to
the $p$-blocks of $G$; see~\cite{Erdmann:Blocks}.
In particular, the Ext algebra of $G$ with coefficients in $K$ can be computed
by constructing a minimal projective resolution of the simple module
associated to each vertex of the quiver corresponding to $\mathcal A$.

If an initial segment of a minimal projective resolution is given, then one
needs to compute minimal generating sets of kernels of homomorphisms of free
right-$\mathcal A$ modules, in order to compute the next term of the
resolution.
Three main computational approaches have been considered: Gr\"obner bases with
respect to well-orderings in path algebras, linear algebra, and standard bases
with respect to negative degree orderings in path algebras.

It is well known~\cite{Singular:Book} that kernels of module homomorphisms can
be computed by means of Gr\"obner bases. D.\@ Farkas, C.\@ Feustel, and E.\@
Green~\cite{FarkasFeustelGreen:Synergy} have provided a Gr\"obner basis theory
for path algebra quotients. In~\cite{FeustelGreen:ProjRes}, it was shown
how to obtain minimal generating sets from the Gr\"obner bases, in the case of
basic algebras. See also~\cite{GreenSolbergZacharia:MinimalProj}.

However, since $\mathcal A$ is of finite dimension as a $K$-vector space, the
problem can also be solved by linear algebra. This approach was exploited by J.\@
Carlson~\cite{Carlson:d64gps} in the first complete computation of the modular
cohomology rings of all groups of order $64$.

D.\@ Green~\cite{Green:Habil} suggested to make use of \emph{negative}
monomial orderings; in this setting, one would talk about \emph{standard
  bases} rather than Gr\"obner bases. His standard basis theory for sub-modules
$M$ of a free right-$\mathcal A$ module also takes into account containment in
the radical of $M$. The standard bases thus obtained are called
\emph{heady}. Green shows that a minimal generating set of $M$ can easily be
read off of a heady standard basis of $M$.
In examples, standard bases for group algebras are often much smaller than
Gr\"obner bases, as was pointed out in~\cite{Green:StandardVsGroebner}.

D.\@ Green implemented his algorithms in C and has shown in many practical
computations that heady standard bases often perform better than linear
algebra. The author of this paper used Green's programs in a modular group
cohomology package~\cite{SimonsProg} for the open source computer algebra
system Sage~\cite{Sage}, which resulted in the first complete computation of
the modular cohomology rings of all groups of order $128$~\cite{128gps} and of
several bigger groups, for different primes, including the mod-$2$ cohomology
of the third Conway
group~\cite{KingGreenEllis:Co3}.

Since we use negative degree orderings, our approach is in the spirit
of~\cite{Green:Habil}. But~\cite{Green:Habil} focuses on modular group rings
of groups of prime power order, which corresponds to basic algebras whose
quivers have a single vertex. The first purpose of this paper is to lift that
restriction and consider general basic algebras.

The algorithm for the computation of heady standard bases in~\cite{Green:Habil}
is of Buchberger type. In particular, performance suffers when too many
useless critical pairs are considered, \emph{i.e.}, critical pairs whose
S-polynomials reduce to zero. The second purpose of this paper is to show that
one can use an \FF algorithm~\cite{Faugere:F5} for right-modules over basic
algebras. With the \FF algorithm, the number of useless critical pairs can be
drastically reduced.

In fact, the \FF algorithm computes a generalisation of heady standard
bases, which we call \emph{signed standard bases}. The third purpose
of this paper is to show: When a signed standard basis of a sub-module
$M$ of a free right-$\mathcal A$ module is known with respect to a
negative degree monomial ordering, then one can not only read off of
it a minimal generating set of $M$, but one can read off a $K$-vector
space basis for each \emph{Loewy layer} of $M$. So, the \FF algorithm
is not only more efficient than the Buchberger style algorithm
from~\cite{Green:Habil}, but it yields more information. Loewy layers
are defined as follows.

\begin{defn}[\cite{Benson:I}]
  The \emph{radical} $\Rad(M)$ of $M$ is the intersection of all the
  maximal submodules of $M$.
  Define $\Rad^0(M) = M$ and $$\Rad^d(M) =
  \Rad\left(\Rad^{d-1}(M)\right)$$ for $d=1,2,...$. The $d$-th \emph{Loewy layer} is
  $\Rad^{d-1}(M)/\Rad^d(M)$.
\end{defn}
In particular, the first Loewy layer $M/\Rad(M)$ is the head of $M$, and a
minimal generating set of $M$ as a right-$\mathcal A$ module corresponds to
a $K$-vector space basis of $M/\Rad(M)$.

This paper is organised as follows. In Section~\ref{sec:orderings}, we recall
some notions from standard basis theory, adapted to path algebra quotients
(not necessarily finite dimensional).

In Section~\ref{sec:Module_orderings}, we show how Buchberger's algorithm
looks like for right-modules over path algebra quotients. The main difference
is that one replaces S-polynomials by so-called \emph{topplings}. Since our
focus is on the \FF algorithm, we skip most proofs in that section and refer
to~\cite{Green:Habil}. We conclude the section by discussing the chain
criterion in the context of topplings.

Section~\ref{sec:signature} is devoted to \emph{signed standard bases} of
right-modules over path algebra quotients. The main result in that section is
Theorem~\ref{thm:sigBuchberger}, providing a criterion for detection of signed
standard bases that combines Faug\`ere's \FF and rewritten criteria. The
criterion is used in Algorithm~\ref{alg:F5}. 

The \FF criterion helps to discard many critical pairs whose S-polynomials
would reduce to zero in the Buchberger algorithm. Nonetheless, in general,
zero reductions can not be totally avoided. But by an idea
from~\cite{ArriPerry:F5revised}, any occurring zero reduction can be used to
improve the criterion and thus helps to avoid some other zero reductions.

Algorithm~\ref{alg:F5} may not terminate for general path algebra
quotients, but \emph{if} it terminates, then it returns a signed
standard basis. We did not try to find the most general conditions for
termination (in particular we do not prove that termination is granted
for all noetherian path algebra quotients), since our work is
motivated by the study of basic algebras, for which termination of
Algorithm~\ref{alg:F5} is clear.

The final Section~\ref{sec:Loewy} shows how signed standard bases for negative
degree monomial orderings can be used
to compute Loewy layers.


\section{Monomial orderings for quotients of path algebras}
\label{sec:orderings}

Let $\mathcal P$ be a path algebra over a field $K$, given by a
finite quiver, $Q$. Let $x_1,...,x_n$ be generators of $\mathcal P$
corresponding to the arrows $\alpha_1,...,\alpha_n$ of $Q$. To each
vertex $v$ of $Q$ corresponds an idempotent $1_v\in \mathcal P$, such
that
$$1_v\cdot x_i=
\begin{cases}
  x_i &  \text{if $\alpha_i$ starts at $v$}\\
  0   &  \text{otherwise}
\end{cases}
$$
and
$$x_j\cdot 1_v=
\begin{cases}
  x_j &  \text{if $\alpha_j$ ends at $v$}\\
  0   &  \text{otherwise}
\end{cases}
$$

Multiplication in $\mathcal P$ corresponds to concatenation of paths
in $Q$. In particular, the directed path in $Q$ are in one-to-one
correspondence with the power products of the generators
$x_1,...,x_n$, to which we refer to as the \emph{monomials} of
$\mathcal P$, the length of the path defining the degree of the
monomial. If a monomial $b$ of $\mathcal P$ corresponds to a path in
$Q$ with endpoint $v$ and a monomial $c$ of $\mathcal P$ corresponds
to a path in $Q$ with startpoint $w\not=v$, then we define $b\cdot
c=0$.  The trivial path consisting of vertex $v$ corresponds to the
monomial $1_v$. The set $\Mon(\mathcal P)$ of monomials of $\mathcal
P$ forms a basis of $\mathcal P$ as a $K$-vector space.

Let $\mathcal A$ be a quotient algebra of $\mathcal P$. In this paper, we
focus on finitely generated sub-modules $M$ of a free right-$\mathcal A$
module $\mathcal A^r$ of rank $r$. In our main application, $\mathcal A$ is a
\emph{basic algebra}, hence finite dimensional over $K$. However, for the
moment we make no restrictions.  For the following notions related with the
theory of standard bases, we adapt the notions from~\cite{Singular:Book}.

\begin{defn}
  A \emph{monomial ordering} of $\mathcal P$ is a total ordering $>$
  on $\Mon(\mathcal P)$, such that
  $$
    b_1 > b_2 \Longrightarrow b_1\cdot b> b_2\cdot b \text{ and } b'\cdot b_1>b'\cdot b_2
  $$
  for all $b_1,b_2,b,b'\in \Mon(\mathcal P)$ such that
  $b_1\cdot b$, $b_2\cdot b$, $b'\cdot b_1$ and $b'\cdot b_2$ are
  all non-zero.
\end{defn}

\begin{defn} \mbox{}
  \begin{itemize}
  \item A monomial ordering of $\mathcal P$ is called \emph{positive},
    or \emph{global}, if $b>1_v$ for any vertex $v$ of $Q$ and any
    monomial $b\in \Mon(\mathcal P)$ with $\deg(b)>0$.
  \item A monomial ordering of $\mathcal P$ is called \emph{negative},
    or \emph{local}, if $b<1_v$ for any vertex $v$ of $Q$ and any
    monomial $b\in \Mon(\mathcal P)$ with $\deg(b)>0$.
  \item A monomial ordering of $\mathcal P$ is called a \emph{positive
      degree ordering} (resp.\@ a \emph{negative degree ordering}) if
    $\deg(b_1)>\deg(b_2)$ implies $b_1>b_2$ (resp.\@ $b_2>b_1$), for
    all $b_1,b_2\in\Mon(\mathcal P)$.
  \end{itemize}
\end{defn}

Given a monomial ordering $>$, each element $p\in\mathcal P$ can be
uniquely written as
$$
p = \alpha_1b_1 + ... + \alpha_kb_k
$$
with $\alpha_1,...,\alpha_k\in K\setminus\{0\}$ and monomials
$b_1>b_2>...>b_k$. If $p\not=0$, we define
\begin{itemize}
\item $\LM(p) = b_1$, the \emph{leading monomial} of $p$,
\item $\LC(p) = \alpha_1$, the \emph{leading coefficient} of $p$,
\item $\LT(p) = \alpha_1b_1$, the \emph{leading term} or \emph{head} of $p$,
\item $\tail(p) = p-\LT(p)$, the \emph{tail} of $p$.
\end{itemize}

Since we are focusing on right modules, our notion of divisibility of
monomials prefers one side as well:
\begin{defn}
  Let $b,c\in\Mon(\mathcal P)$ be monomials. We say that $b$ divides
  $c$ (denoted $b\divides c$), if there is a monomial $b'\in\Mon(\mathcal P)$
  such that $b\cdot b'=c$.
\end{defn}

\begin{defn}
  Let $\psi\co \mathcal P \to \mathcal A$ be the quotient map. Then
  $\stdmon_{\mathcal A}(\mathcal P)\subset \Mon(\mathcal P)$ is formed by all
  monomials that are not leading monomials of elements of $\ker(\psi)$.
\end{defn}
We assume in this paper that the standard monomials are known. Note that the
standard monomials could alternatively be obtained by linear algebra or from a two-sided
standard basis of $\ker(\psi)$, if $\mathcal A$ is finite dimensional. One
easily sees that $\psi$ is injective on $\stdmon_{\mathcal A}(\mathcal P)$,
and that $\mathcal B_>(\mathcal A) = \psi\left(\stdmon_{\mathcal
  A}(\mathcal P)\right)$ is a basis of $\mathcal A$ as a $K$-vector space. We
call it the \emph{preferred basis}, following~\cite{Green:Habil}, and we call
its elements the \emph{monomials} of $\mathcal A$.

\begin{defn}
  We define the \emph{lift} $\lambda(b)$ of a monomial $b\in \mathcal
  B_>(\mathcal A)$ of $\mathcal A$ as the unique element of $\stdmon_{\mathcal
    A}(\mathcal P)$ with $\psi\left(\lambda(b)\right)=b$.
\end{defn}

We can now define divisibility of elements of the preferred basis and
an ordering on the preferred basis as follows.
\begin{defn}
  Let $b,b'\in \mathcal B_>(\mathcal A)$. We say that $b$
  \emph{strictly divides} $b'$ (or $b\strdiv b'$) if and only if
  $\lambda(b)\divides \lambda(b')$, and define $b>b'$ if and only
  if $\lambda(b)> \lambda(b')$. We define $\deg(b)= \deg\left(\lambda(b)\right)$.
\end{defn}
Note that the existence of some $s\in \mathcal A$ or even $s \in
\mathcal B_>(\mathcal A)$ with $b\cdot s = b'$ does not necessarily
imply that $b\strdiv b'$. However, strict divisibility is a transitive
relation.

We call the induced ordering on the preferred basis a \emph{monomial
  ordering on $\mathcal A$}, inheriting the properties ``positive'',
``negative'' or ``degree ordering'' from the monomial ordering on
$\mathcal P$.
\begin{defn}
  A monomial ordering on $\mathcal A$ is \emph{admissible}, if there
  is no infinite strictly decreasing sequence of monomials of
  $\mathcal A$.
\end{defn}

Since any element of $f\in \mathcal A$ can be uniquely written as
$$
f = \alpha_1b_1 + ... + \alpha_kb_k
$$
with $\alpha_1,...,\alpha_k\in K\setminus\{0\}$ and $b_1,...,b_k\in \mathcal
B_>(\mathcal A)$ with $b_1>b_2>...>b_k$, we can define the leading
monomial $\LM(f)=b_1$, the leading coefficient $\LC(f)=\alpha_1$, the
leading term $\LT(f)=\alpha_1b_1$ and the tail $\tail(f)=f-\LT(f)$ of $f$,
provided $f\not=0$.

With leading monomials being defined in $\mathcal A$, we can define
another notion of divisibility on $\mathcal B_>(\mathcal A)$, that is
weaker than strict divisibility.
\begin{defn}
  Let $b,b'\in \mathcal B_>(\mathcal A)$. We say that $b$ divides $b'$
  (or $b\divides b'$) if and only if there is some $c\in \mathcal
  B_>(\mathcal A)$ such that $\LM(b\cdot c)=b'$.
\end{defn}
Obviously $b\strdiv b'$ implies $b\divides b'$. Note that $b\divides
b'$ does not necessarily mean that there is some $f\in \mathcal A$
with $b\cdot f=b'$.

Divisibility and strict divisibility are interrelated by the following
notion, that we adapt from~\cite{Green:Habil}.
\begin{defn}\mbox{}
  Let $b \in \mathcal B_>(\mathcal A)$ and $t\in \mathcal B_>(\mathcal A)\cup\{0\}$.
  \begin{enumerate}
  \item If there is a $c\in\mathcal B_>(\mathcal A)$ such that either
    $\LM(b\cdot c)=t\not=0$ and $b\not\strdiv t$, or
    $b\cdot c=0=t$, then $t$ is called a \emph{toppling} of $b$
    with \emph{cofactor} $c$.
  \item If $c\in\mathcal B_>(\mathcal A)$ is not cofactor of a
    toppling of $b$ then we call $c$ a
    \emph{small cofactor} of $b$.
  \item Let $t$ be a toppling of $b$ with a cofactor $c$. Assume that
    all $c'\in \mathcal B_>(\mathcal A)$ with $c'\strdiv c$ and
    $c'\not=c$ are small cofactors of $b$. Then $t$ is called a
    \emph{minimal} toppling of $b$.
  \end{enumerate}
\end{defn}
We make two easy observations:

\begin{rk}
  Let $b,c,c'\in \mathcal B_>(\mathcal A)$. Then $c$ is a small
  cofactor of $b$ and $c'$ is a small cofactor of $b\cdot c$ if and
  only if $c\cdot c'$ is a small cofactor of $b$.
\end{rk}

\begin{rk}\label{rk:decompose}
  For any $b,b'\in \mathcal B_>(\mathcal A)$ with $b\divides b'$, there is a
  finite sequence $b=b_0,b_1,b_2,...,b_m\in \mathcal B_>(\mathcal A)$
  such that $b_i$ is a minimal toppling of $b_{i-1}$, for all
  $i=1,...,m$, and $b_m\strdiv b'$.
\end{rk}

\begin{lem}\label{lem:SmallCofactorA}
  Let $a,a',b\in \mathcal B_>(\mathcal A)$ with $a>a'$, so that $b$ is a small
  cofactor of $a$, and $a'\cdot b\not=0$. Then $\LM(a\cdot b)>\LM(a'\cdot b)$.
  In particular, if $f\in \mathcal A$ and $b\in \mathcal B_>(\mathcal A)$ is a
  small cofactor of $\LM(f)$, then $\LM(f\cdot b) = \LM(f)\cdot b$.
\end{lem}
\begin{proof}
  Let $\tilde a=\lambda(a)$, $\tilde a'=\lambda(a')$ and $\tilde
  b=\lambda(b)$. By definition of the monomial ordering on $\mathcal A$, we
  have $\tilde a'<\tilde a$ and thus $\tilde a'\cdot \tilde b<\tilde a\cdot
  \tilde b$. Since $b$ is a small cofactor of a, we have $\tilde a\cdot \tilde
  b\in \stdmon_{\mathcal A}(\mathcal P)$.

  If $\tilde a'\cdot \tilde b\in \stdmon_{\mathcal A}(\mathcal P)$ then
  $\LM(a'\cdot b) = \psi(\tilde a')\cdot b = \psi(\tilde a'\cdot \tilde b)
  <\psi(\tilde a\cdot \tilde b)$.
  Otherwise, let $\psi(\tilde a'\cdot \tilde b) = \psi(\tilde a')\cdot b =
  \alpha_1b_1 + ... + \alpha_kb_k$ with $\alpha_1,...,\alpha_k\in
  K\setminus\{0\}$ and $b_1,...,b_k\in \mathcal B_>(\mathcal A)$ with
  $b_1=\LM\left(\psi(\tilde a'\cdot \tilde b)\right)$, and let $\tilde b_i =
  \lambda(b_i)$ for $i= 1,...,k$. We have $k>0$, since $a'\cdot b\not=0$ by
  hypothesis.

  Since $z = \tilde a'\cdot \tilde b - \left( \alpha_1\tilde b_1 + ... +
  \alpha_k\tilde b_k\right)\in \ker(\psi)$ and since the standard monomials
  $\tilde b_1,...,\tilde b_k$ are by definition not leading monomials of any
  element of $\ker(\psi)$, it follows $\tilde b_i<\LM(z) = \tilde a'\cdot
  \tilde b <\tilde a\cdot \tilde b$ for $i=1,...,k$. Hence, $b_1 = \LM\left(\psi(\tilde a'\cdot
  \tilde b)\right) = \LM(a'\cdot b) < \LM(\psi(\tilde a\cdot \tilde b)) =
  \LM(a\cdot b)$.

  For the last statement of the lemma, we already know that $\LM(f\cdot
  b)=\LM(\LM(f)\cdot b)$. There remains to observe that $\LM(f)\cdot
  b=\psi(\tilde a\cdot \tilde b)$, since $b$ is not cofactor of a toppling of
  $\LM(f)$. Hence, $\LM(\LM(f)\cdot b)=\LM(f)\cdot b$.
\end{proof}

\section{Monomial orderings for right modules quotients of path algebras}
\label{sec:Module_orderings}

We use the same notations as in the previous section and are now
studying the free right-$\mathcal A$ module $F=\mathcal A^r$ of rank
$r$. Let $\mathfrak v_1,...,\mathfrak v_r$ be free generators of $F$.

\begin{defn}
  Let a monomial ordering $>$ on $\mathcal A$ be given, so that the
  notion of leading monomial is defined in the quotient $\mathcal A$
  of $\mathcal P$.
  \begin{enumerate}
  \item The set of \emph{monomials} $\Mon(F)$ of $F$ is the set of all $\mathfrak v_i\cdot b$
    with  $i=1,...,r$ and $b\in \mathcal B_>(\mathcal A)$.
  \item A total ordering $>$ on $\Mon(F)$ is called a \emph{monomial
      ordering compatible with $>$ on $\mathcal A$}, if it satisfies
    \begin{enumerate}
    \item $b_1>b_2 \Longrightarrow \mathfrak v_i\cdot b_1>\mathfrak v_i\cdot b_2$
    \item $\mathfrak v_i\cdot b_1>\mathfrak v_j\cdot b_2
      \Longrightarrow \mathfrak v_i\cdot b_1\cdot b > \mathfrak
      v_j\cdot \LM(b_2\cdot b)$ or $b_2\cdot b=0$
    \end{enumerate}
    for all $i,j=1,...,r$ and all $b_1,b_2,b\in \stdmon(\mathcal A)$
    so that $b$ is a small cofactor of $b_1$.
  \end{enumerate}
\end{defn}

Since any $f\in F$ can be uniquely written as a linear combination of
monomials of $F$, we obtain the notions of \emph{leading monomial}
$\LM(f)$, \emph{leading coefficient} $\LC(f)$, \emph{leading term}
$\LT(f)$ and \emph{tail} $\tail (f)$ as we did for elements of
$\mathcal A$.

\begin{defn}
  Let $\mathfrak v_i\cdot b_1, \mathfrak v_j\cdot b_2\in \Mon(F)$. We
  say that $\mathfrak v_i\cdot b_1$ \emph{divides} (resp.\@
  \emph{strictly divides}) $\mathfrak v_j\cdot b_2$, denoted
  $\mathfrak v_i\cdot b_1\divides \mathfrak v_j\cdot b_2$ (resp.\@
  $\mathfrak v_i\cdot b_1\strdiv \mathfrak v_j\cdot b_2$), if and only
  if $i=j$ and $b_1\divides b_2$ (resp.\@ $i=j$ and $b_1\strdiv b_2$).
\end{defn}
\begin{lem}\label{lem:SmallCofactor}
  Let $f\in F$ with $\LM(f) = \mathfrak v_i\cdot b_f$, and let $b\in
  \mathcal B_>(\mathfrak A)$ be a small cofactor of $b_f$. Then
  $\LM(f\cdot b) = \LM(f)\cdot b$.
\end{lem}
\begin{proof}
  Lemma~\ref{lem:SmallCofactorA} can be applied component-wise.
\end{proof}

\begin{defn}
  Let $G$ be a finite subset of $F\setminus\{0\}$. An
  element $f\in F$ has a \emph{standard representation} with respect
  to $G$ if and only if $f$ can be written as
  $$
  f = \sum_{i=1}^m \alpha_i g_i\cdot c_i
  $$
  such that $\alpha_i\in K\setminus\{0\}$, $g_i\in G$, $c_i\in \mathcal
  B_>(\mathcal A)$ is a small cofactor of $\LM(g_i)$, and
  $\LM(g_i)\cdot c_i = \LM(g_i\cdot c_i) \le \LM(f)$, for all
  $i=1,...,m$.
%
\end{defn}

\begin{defn}
  A \emph{normal form} $\NF$ on $F$ assigns to any $f\in F$ and any
  finite subset $G\subset F\setminus\{0\}$ an element $\NF(f,G)\in F$,
  such that the following holds.
  \begin{enumerate}
  \item $\NF(0,G)=0$.
  \item If $\NF(f,G)\not=0$ then $\LM(g)\not\strdiv
    \LM\left(\NF(f,G)\right)$, for all $g \in G$.
  \item $f-\NF(f,G)$ has a standard representation with respect to
    $G$.
  \end{enumerate}
\end{defn}

\begin{algorithm}[H]
  \DontPrintSemicolon
  \KwData{$f\in F$ and a finite subset $G\subset F\setminus \{0\}$}
  \KwResult{$\NF(f,G)$}
  \Begin{
    $f_r\longleftarrow f$\;
    \While{$f_r\not=0$ and there is some $g\in G$ with $\LM(g)\strdiv \LM(f_r)$}{
      Let $c\in \mathcal B_>(\mathcal A)$ be the small cofactor of $\LM(g)$ with
      $\LM(f_r) = \LM(g)\cdot c$\;
      $f_r \longleftarrow f_r-\frac{\LC(f_r)}{\LC(g)}g\cdot c$\;
    }
    \Return $f_r$\;
  }
  \caption{A normal form algorithm}
  \label{alg:NF}
\end{algorithm}
\begin{lem}\label{lem:algNF}
  If the monomial ordering on $\mathcal A$ is admissible, then
  Algorithm~\ref{alg:NF} computes a normal form on $F$.
\end{lem}
\begin{proof}
  In the while-loop of Algorithm~\ref{alg:NF}, $\LM(f_r)$ strictly
  decreases, since $\LM(g\cdot c) = \LM(g)\cdot c = \LM(f_r)$, by
  Lemma~\ref{lem:SmallCofactor}. Since $F$ is of finite rank and since
  the monomial ordering on $\mathcal A$ is admissible,
  Algorithm~\ref{alg:NF} terminates in finite time.

  By construction, the leading monomial of the returned element $f_r$
  is not strictly divisible by the leading monomial of any $g\in G$.
  Moreover, $f-f_r$ has a standard representation with respect to $G$,
  since all cofactors in the algorithm are small, and since in the
  while loop holds $\LM(g\cdot c) = \LM(g)\cdot c$ by
  Lemma~\ref{lem:SmallCofactor} and $\LM(g)\cdot c=\LM(f_r)\le \LM(f)$
  by construction.
\end{proof}

\subsection{Standard bases for right modules over quotients of path algebras}

Let $\mathcal P$, $\mathcal A$ and $F$ be as in the previous
section. We are now studying submodules $M\subset F$ and assume that
we have an admissible monomial ordering on $\mathcal A$.
\begin{defn}
  Let $G\subset M\setminus \{0\}$ be a finite subset. An element $f\in
  F\setminus\{0\}$ is called \emph{reducible with respect to $G$}, if there is
  some $g\in G$ such that $\LM(g)\strdiv \LM(f)$. Otherwise, it is called
  \emph{irreducible with respect to $G$}.

  A finite subest $G\subset M\setminus \{0\}$ is called \emph{interreduced},
  if every $g\in G$ is irreducible with respect to $G\setminus \{g\}$.

  A finite subest $G\subset M\setminus \{0\}$ is called a
  \emph{standard basis} of $M$, if every $f\in M\setminus \{0\}$ is
  reducible with respect to $G$.
\end{defn}


\begin{lem}
  If $G$ is a standard basis of $M$, then every element of $M$ has a
  standard representation with respect to $G$. In particular, $G$
  generates $M$ as a right-$\mathcal A$ module, and $f\in M$ if and
  only if $\NF(f,G)=0$.
\end{lem}
\begin{proof}
  Let $f\in F$. Since $G\subset M$, $\NF(f,G)=0$ implies that $f\in
  M$.

  Now, assume that $f\in M$.  Since $f\in M$, the element $f_r$ in the
  while-loop of Algorithm~\ref{alg:NF} is in $M$ as well. Hence, by
  definition of a standard basis, there is $g\in G$ whose leading
  monomial is a strict divisor of the leading monomial of
  $f_r$. Hence, the algorithm will continue until $f_r=0$.
\end{proof}

\begin{lem}\label{lem:make_interreduced}
  Let $G\subset M\setminus \{0\}$ be a finite subset. If $f\in
  M\setminus \{0\}$ is reducible with respect to $G$ then it is reducible with respect to 
  $$
  \left\{\NF\left(g, G\setminus \{g\}\right)\right\} \cup G\setminus\{g\}
  $$
  for all $g\in G$.
\end{lem}
\begin{proof}
  If $\LM(f)$ is strictly divisible by the leading monomial of some
  element of $G\setminus \{g\}$ or if $\LM\left(\NF\left(g, G\setminus
      \{g\}\right)\right) = \LM(g)$, then there is nothing to show.

  Otherwise, $\LM(g)\strdiv \LM(f)$ and there is some $g'\in
  G\setminus \{g\}$ with $\LM(g')\strdiv \LM(g)$. Hence,
  $\LM(g')\strdiv \LM(f)$, since strict divisibility is transitive.
\end{proof}

\begin{cor}\label{cor:interred}
  For any finite subset $G\subset M\setminus\{0\}$, there is an
  \emph{interreduced} finite subset $\interred(G)\subset
  M\setminus\{0\}$ such that if $f$ is reducible with respect to $G$
  then $f$ is reducible with respect to $\interred(G)$, for all $f\in
  F$.
\end{cor}
\begin{proof}
  If $G$ is not interreduced, there is some $g\in G$ such that
  $g'=\NF(g,G\setminus \{g\})\not=g$ and $\LM(g')<\LM(g)$.  We replace
  $G$ by $\{g'\}\cup G\setminus \{g\}$. By the preceding Lemma, the
  change of $G$ does not decrease the set of elements of $F$ that are
  reducible with respect to $G$.

  Since the leading monomial strictly decreases and we only have
  finitely many monomials, we obtain an $\interred(G)$ after finitely
  many steps.
\end{proof}

\begin{defn}
  We say that a finite subset $G\subset M\setminus \{0\}$
  satisfies \emph{property (T)}, if it is interreducted, and $g\cdot
  c$ has a standard representation with respect to $G$, for every
  $g\in G$ and the cofactor $c$ of every minimal toppling of $\LM(g)$.
\end{defn}

\begin{thm}[\cite{Green:Habil}]\label{thm:criterion}
  If a finite subset $G\subset M\setminus \{0\}$ generates $M$
  as a right-$\mathcal A$ module and has property (T), then it is a
  standard basis of $M$.
\end{thm}
Our main result, Theorem~\ref{thm:sigBuchberger}, is a generalisation of
Theorem~\ref{thm:criterion}. Therefore, for the sake of brevity, we do not
include a proof of Theorem~\ref{thm:criterion}. Note that~\cite{Green:Habil}
obtains a similar result for so-called \emph{heady} standard bases, which in
turn is a special case of signed standard bases.

\subsection{A short account on the chain criterion}

We now focus on the case that $\mathcal A$ is finite dimensional over
$K$. Then, there is a standard basis of $M$, simply since $\Mon(F)$ is
finite. But finding a standard basis by an enumeration of leading
monomials would certainly not be very
efficient. Theorem~\ref{thm:criterion} provides a more efficient
algorithm for the computation of a standard basis of $M$, similar to
Buchberger's algorithm.

\begin{algorithm}[H]
  \DontPrintSemicolon
  \KwData{$G=\{g_1,...,g_k\}$, generating $M$ as a right-$\mathcal A$ module.}
  \KwResult{An interreduced standard basis of $M$.}
  \Begin{ \While{ There is some $g\in G$ and a cofactor $c$
      of a minimal toppling of $\LM(g)$ such that $g\cdot c$ has
      no standard representation with respect to $G$ }
    {
      $G\longleftarrow \interred\left(G\cup \left\{\NF(g\cdot c, G)\right\}\right)$\;
    }
    \Return $G$\;
  }
  \caption{A Buchberger style computation of a standard basis}
  \label{alg:Buchberger}
\end{algorithm}

\begin{lem}\label{lem:Buchberger}
  Algorithm~\ref{alg:Buchberger} computes an interreduced standard
  basis of $M$, in the case that $\mathcal A$ is finite dimensional
  over $K$.
\end{lem}
\begin{proof}
  In the while-loop of the algorithm, the number of monomials of $M$
  that are strictly divisible by the leading monomial of an element of
  $G$ strictly increases. Since there are only finitely many
  monomials, the computation terminates in finite time. The result is
  a standard basis, by Corollary~\ref{cor:interred} and
  Theorem~\ref{thm:criterion}.
\end{proof}

Algorithm~\ref{alg:Buchberger} is certainly more efficient than an
attempt to directly enumerate all leading monomials of a sub-module
$M\subset F$. However, there is a common problem in the computation of
standard bases: In the head of the while-loop in
Algorithm~\ref{alg:Buchberger}, one needs to test whether $g\cdot c$
has a standard representation with respect to $G$, \emph{i.e.},
whether $\NF(g\cdot c,G)=0$. If $\NF(g\cdot c,G)=0$, then the pair
$(g,c)$ does not contribute to the standard basis. Since the
computation of $\NF(g\cdot c,G)$ is a non-trivial task, it would be
nice to have a criterion that disregards the pair $(g,c)$
\emph{without} computation of a normal form.

There are several criteria known from the computation of Gr\"obner
bases in the \emph{commutative} case, such as Buchberger's product or
chain criteria~\cite{Buchberger:Criteria}, or Faug\`ere's \FF and
rewritten criteria~\cite{Faugere:F5}.

We discuss how these criteria apply in our non-commutative
context. First of all, since we consider modules of rank greater than
one, the product criterion would not hold, even in the commutative
case~\cite[Remark~2.5.11]{Singular:Book}. Faug\`ere's criteria are the
subject of the next section. Here, we argue that in fact we are
already using some kind of chain criterion.

Our rings being highly non-commutative is not the only problem with
using existing criteria: The computation of standard bases usually is
based on $S$-polynomials, which are not even mentioned in
Algorithm~\ref{alg:Buchberger}. However, a different point of view
shows that $S$-polynomials are hidden in the notion of a
toppling. This point of view appears in more detail
in~\cite{Green:Habil}, by describing the computation of normal forms
as ``two-speed reduction''.

Namely, one could model $F=\mathcal A^r$ as $\mathcal
P^r/\ker(\psi)^r$, using a two-sided standard basis $S$ of
$\ker(\psi)$. Computing a standard basis of a sub-module $M\subset F$
could be done by lifting the generators of $M$ to elements of
$\stdmon_{\mathcal A}(\mathcal P)^r$, and adding to the lifted
generators a copy of $S$ in each component of the free module.

Let $\tilde g\in \mathcal P^r$ be the lift of a generator $g$ of $M$ with
$\LM(\tilde g)$ in the $i$-th component, and $s\in S$. If there are monomials
$a,b,c$ of $\mathcal P$ such that $\mathfrak v_i\cdot a\cdot b = \LM(\tilde
g)$ and $b\cdot c=\LM(s)$, then the $S$-polynomial of $\tilde g$ and
$\mathfrak v_i\cdot s$ is
$$
\tilde g\cdot c - \frac{\LC(\tilde g)}{\LC(s)}\mathfrak v_i\cdot a\cdot s
$$

Note that this only occurs if $\LM(\tilde g)\cdot c$ does not belong to
$\stdmon_{\mathcal A}(\mathcal B)^r$. When we now apply $\psi^r$, then the
$S$-polynomial is mapped to $g\cdot \psi(c)$: The second summand vanishes,
because $s$ belongs to a two-sided standard basis of $\ker(\psi)$. Hence, the
$S$-polynomial in $\mathcal P^r$ corresponds to multiplying an element of $F$
with a toppling $\psi(c)$ of its leading monomial.

Let $g,\tilde g, s,a,b,c$ be as above. Assume that
there is $s\not=s'\in S$ and monomials $a',b',c',d$
of $\mathcal P$ such that $\mathfrak v_i\cdot a'\cdot b' =
\LM(\tilde g)$ and $b'\cdot c'=\LM(s')$ and $a'\cdot
\LM(s')\cdot d = a\cdot \LM(s)$. Then the chain criterion
says that the $S$-polynomial of the pair $(\tilde g,\mathfrak v_i\cdot
s)$ does not need to be considered.

In our context, if we find $s'$ as above, then $c = c'\cdot d$ is not cofactor
of a minimal toppling of $\LM(g)$, since $c'$ is a cofactor of a toppling as
well. Hence, the chain criterion corresponds to the fact that we only consider
\emph{minimal} topplings in Algorithm~\ref{alg:Buchberger}.

\section{Signatures}
\label{sec:signature}

As before, let $M$ be a finitely generated sub-module of a free
right-$\mathcal A$ module $F$ of rank $r$. We fix a finite ordered
generating set $\{\hat g_1,...,\hat g_m\}$ of $M$. Let $E=\mathcal
P^m$ be a free right-$\mathcal P$ module with free generators
$\mathfrak e_1,...,\mathfrak e_m$. We fix a monomial ordering on
$\mathcal P$ that induces an admissible monomial ordering on $\mathcal
A$, and fix a monomial ordering on $E$
compatible with the monomial ordering on $\mathcal P$.

By applying $\psi$ and evaluating
at the generating set of $M$, we obtain a map $ev\co E\to M$ with
$ev(\mathfrak e_i\cdot c) = \hat g_i\cdot \psi(c)$, for all monomials
$\mathfrak e_i\cdot c$ of $E$.
\begin{defn}
  A \emph{signed element} of $M$ is a pair $(f_u, s)$ with $f_u\in M$ and
  $s\in \Mon(E)$, such that there is some $\tilde f\in E$ with
  $\LM(\tilde f) = \sigma$ and $ev(\tilde f)=f_u$.

  For a signed element $f = (f_u,s)$ of $M$, we define the
  \emph{unsigned element} $\poly(f)=f_u$ and the \emph{signature}
  $\sig(\overline f)=s$.

  A \emph{signed subset} of $M$ is a set formed by signed elements of
  $M$.
\end{defn}

If $b$ is a monomial of $\mathcal A$ and $f$ is a signed element of
$M$, then for all $\tilde b\in \mathcal P$ with
$\psi(\tilde b)=b$, $(\poly(f)\cdot b, \LM(\sig(f)\cdot \tilde b))$ is a signed
element of $M$ as well.  For simplicity, we write $\LM(f) = \LM(\poly(f))$
and $\LC(f) = \LC(\poly(f))$.

\begin{defn}
  Let $f\in M$ and $s\in \Mon(E)$. If there is a signed element $g$ of $M$
  with $f = \poly(g)$ and $\sig(g)<s$, then $f$ is \emph{dominated by $s$}.  A
  signed element $f$ of $M$ is \emph{suboptimal}, if $\poly(f)$ is dominated
  by $\sig(f)$.
\end{defn}

\begin{defn}
  Let $G$ be a finite signed subset of $M$, let $f\in F$, and let $s\in \Mon(E)$.
  An \emph{$s$-standard representation with respect to $G$} of $f$ is a list
  of triples $(\alpha_1,g_1,c_1),..., (\alpha_k,g_k,c_k)$, where $\alpha_i\in
  K$, $g_i\in G$, and $c_i\in \mathcal B_>(\mathcal A)$ is a small cofactor of
  $\LM(g_i)$, for $i=1,...,k$, such that
  $$
  \poly(f) = \sum_{i=1}^k \alpha_i\poly(g_i)\cdot c_i
  $$ 
  and $s>\LM\left(\sig(g_i)\cdot \lambda(c_i)\right)\not=0$ for $i=1,...,k$.
\end{defn}

Note that we do not bound $\LM(\poly(g_i)\cdot c_i)$ in terms of
$f$. For simplicity, if $f$ is a signed element of $M$, then a
\emph{standard representation} of $f$ shall denote a $\sig(f)$-standard
representation of $\poly(f)$.
Clearly, if a signed element $f$ of $F$ has a standard representation with
respect to a signed subset of $M$, then it is a suboptimal signed element of
$M$, and $\sig(f)$ is the leading monomial of an element of $\ker(ev)$.

\begin{defn}
  Let $G$ be a signed subset of $M\setminus \{0\}$, let $f\in F$,
  and let $s\in \Mon(E)$.
  \begin{itemize}
  \item We say that $f$ is \emph{$s$-reducible} (resp.\@ \emph{weakly}
    $s$-reducible) with respect to $G$, if there is some $g\in G$ and
    a small cofactor $c$ of $\LM(g)$ such that $\LM(g)\cdot c =
    \LM(f)$ and $\sig(g)\cdot \lambda(c)<s$ (resp.\@ $\sig(g)\cdot
    \lambda(c)\le s$).
  \item We say that $f$ is (weakly) $s$-reducible with respect to $M$,
    if there is some signed element $g$ of $M\setminus \{0\}$ such
    that $f$ is (weakly) $s$-reducible with respect to $\{g\}$.
  \item We say that $f$ is \emph{$s$-irreducible} with respect to $G$, if
    $\poly(f)=0$ or $f$ is not $s$-reducible.
  \end{itemize}
\end{defn}

\begin{lem}\label{lem:reducible_by_irreducible}
  Let $f\in F$ and let $s\in \Mon(E)$. If $f$ is (weakly) $s$-reducible with
  respect to $M$ then there is some signed element $g$ of $M\setminus \{0\}$
  such that $\poly(g)$ is $\sig(g)$-irreducible with respect to $M$ and $f$ is
  (weakly) $s$-reducible with respect to $\{g\}$.
\end{lem}
\begin{proof}
  By definition, there is some signed element $g$ of $M\setminus
  \{0\}$ such that there is a small cofactor $c$ of $\LM(g)$ with
  $\LM(f)=\LM(g)\cdot c$ and $\sig(g)\cdot \lambda(c)<s$ (resp.\@
  $\sig(g)\cdot \lambda(c)\le s$).  We will show that one can choose
  $g$ so that $\poly(g)$ is $\sig(g)$-irreducible with respect to $M$.

  Assume that $g$ is $\sig(g)$-reducible with respect to $M$. Then
  there is some signed element $g'$ of $M\setminus \{0\}$ and a small
  cofactor $c'$ of $\LM(g')$ such that $\LM(g')\cdot c'=\LM(g)$ and
  $\sig(g')\cdot \lambda(c')<\sig(g)$. Since $c'\cdot c$ is a small
  cofactor of $\LM(g')$ and $\sig(g')\cdot \lambda(c'\cdot c) <
  \sig(g)\cdot \lambda(c)<s$, we can replace $g$ by $g'$.

  Since $\lambda(c'\cdot c) < \sig(g)\cdot \lambda(c)<s$ and $<$ is
  supposed to be a well-ordering on $\Mon(E)$, a replacement of $g$ by
  $g'$ can only occur finitely many times. Hence, eventually we find
  $g$ so that $\poly(g)$ is $\sig(g)$-irreducible with respect to $M$.
\end{proof}

\begin{defn}
  A \emph{signed normal form} on $F$ assigns to any $f\in F$, any
  finite signed subset $G$ of $F\setminus\{0\}$ and any $s\in \Mon(E)$
  an element $\NF_s(f,G)\in F$, such that the following holds.
  \begin{enumerate}
  \item If $f=0$ then $\NF_s(f,G)=0$.
  \item $\NF_s(f,G)$ is $s$-irreducible with respect to $G$.
  \item $f-\NF_s(f,G)$ has an $s$-standard representation with respect to $G$.
  \end{enumerate}
\end{defn}

If $f$ is a signed element of $M$, then we implicitly assume $s=\sig(f)$ in
the two preceding definitions, unless stated otherwise. Note that a signed
element $f$ of $M$ is irreducible with respect to $\{f\}$.
We denote $\NF(f,G) = \left(\NF_{\sig(f)}(\poly(f),G), \sig(f)\right)$; it is
easy to see that this is a signed element.

\begin{algorithm}[H]
  \DontPrintSemicolon
  \KwData{$f\in F$, a finite signed subset $G\subset F\setminus \{0\}$ and $s\in \Mon(s)$}
  \KwResult{$\NF_s(f,G)$}
  \Begin{
    $f_r\longleftarrow f$\;
    \While{$\poly(f_r)\not=0$ and there is some $g\in G$ and a small cofactor $c$ of $\LM(g)$
      with $\LM(g)\cdot c = \LM(f_r)$ and $\sig(g)\cdot \lambda(c)<s$}{
      $f_r \longleftarrow f_r-\frac{\LC(f_r)}{\LC(g)}\poly(g)\cdot c$\;
    }
    \Return $f_r$\;
  }
  \caption{A signed normal form algorithm}
  \label{alg:sigNF}
\end{algorithm}
\begin{lem}
  Algorithm~\ref{alg:sigNF} computes a signed normal form on $F$.
\end{lem}
\begin{proof}
  The proof of the lemma is essentially as the proof of
  Lemma~\ref{lem:algNF}.
\end{proof}

\subsection{Signed standard bases}
\label{sec:sig_criterion}

The following definitions were adapted from~\cite{ArriPerry:F5revised}.

\begin{defn}\label{def:signed_stb}
  Let $G$ be a finite signed subset of $M\setminus \{0\}$.
  \begin{enumerate}
  \item If every $g\in G$ is irreducible with respect to $G$, and there is
    no $g'\in G$ with a small cofactor $c$ of $\LM(g')$ such that
    $\LM(g')\cdot c=\LM(g)$ and $\sig(g')\cdot \lambda(c)=\sig(g)$, then $G$
    is called \emph{interreduced}.

  \item Assume that each signed element $f$ of $M\setminus \{0\}$ that
    is $\sig(f)$-irreducible with respect to $M\setminus \{0\}$ is
    weakly $\sig(f)$-reducible with respect to $G$. Then $G$ is a
    \emph{signed standard basis} of $M$.
  \end{enumerate}
\end{defn}

Rephrasing the definition, if $G$ is a signed standard basis of $M$ and $f$ is
a $\sig(f)$-irreducible signed element of $M\setminus \{0\}$, then there is
some $g\in G$ and a small cofactor $c$ of $\LM(g)$ such that $\LM(g)\cdot
c=\LM(f)$ and $\sig(g)\cdot \lambda(c)=\sig(f)$.

\begin{lem}
  Let $G$ be a signed standard basis of $M$, let $f\in F\setminus
  \{0\}$ and let $s\in \Mon(E)$.
  \begin{enumerate}
  \item If $f$ is $s$-reducible with respect to $M$ then it is
    $s$-reducible with respect to $G$.
  \item $f$ is an element of $M$ that is dominated by $s$ if and only
    if $\NF_s(f,G)=0$.
  \item $G'=\{\poly(g)\co g\in G\}$ is a standard basis of $M$.
  \end{enumerate}
\end{lem}
\begin{rk}
  The unsigned normal form $\NF(\poly(g),G')$ can be zero, even if $g\in G$ is
  irreducible with respect to $G$. Hence, $G'$ is usually not a minimal
  standard basis, and not even interreduced.
\end{rk}
\begin{proof}
  \mbox{}
  \begin{enumerate}
  \item By Lemma~\ref{lem:reducible_by_irreducible}, we can assume that there
    is some signed element $h$ of $M\setminus \{0\}$ such that $\poly(h)$ is
    $\sig(h)$-irreducible with respect to $M$, and there is a small cofactor
    $c$ of $\LM(h)$ such that $\LM(h)\cdot c = \LM(f)$ and $\sig(h)\cdot
    \lambda(c)<s$.

    By Definition~\ref{def:signed_stb}, there is some $g$ in $G$ and a
    small cofactor $b$ of $\LM(g)$ such that $\LM(g)\cdot b=\LM(h)$
    and $\sig(g)\cdot \lambda(b)=\sig(h)$. Since $b\cdot c$ is a
    small cofactor of $\LM(g)$, $\LM(g)\cdot b\cdot c=\LM(h)\cdot
    c=\LM(f)$ and $\sig(g)\cdot \lambda(b\cdot c)=\sig(h)\cdot c<s$,
    we find that $f$ is $s$-reducible with respect to $G$.
  \item If $\NF_s(f,G)=0$ then $f$ has an $s$-standard representation
    with respect to $G$. Since $G$ is formed by signed elements of
    $M$, we find $f\in M$, and by the definition of an $s$-standard
    representation it is dominated by $s$.

    If $f\in M$ is dominated by $s$, then the first statement of the
    lemma implies that Algorithm~\ref{alg:sigNF} will end only when
    $f_r=0$. Hence, $\NF_s(f,G)=0$.

  \item Let $f\in M\setminus \{0\}$. We have to show that $f$ is
    reducible with respect to $G'$, \emph{i.e.}, there is some $g\in
    G$ and a small cofactor $c$ of $\LM(g)$ such that $\LM(g)\cdot
    c=\LM(f)$. Since the map $ev\co E\to F$ is surjective, there is
    some $s\in \Mon(E)$ such that $(f,s)$ is a signed element.

    If $f$ is $s$-reducibel with respect to $M$, then it is
    $s$-reducible with respect to $G$ by the first part of the lemma,
    and thus $f$ is reducible with respect to $G'$. Otherwise,
    Definition~\ref{def:signed_stb} ensures that $f$ is reducible with
    respect to $G'$. Hence, $G'$ is a standard basis of $M$.
  \end{enumerate}
\end{proof}

\begin{lem}\label{lem:sig_make_interreduced}
  Let $G$ be a finite signed subset of $M\setminus \{0\}$, $f\in F$ and $s\in \Mon(E)$.
  \begin{enumerate}
  \item If $f$ is (weakly) $s$-reducible with respect to $G$ then it
    is (weakly) $s$-reducible with respect to
    $
    \tilde G = \left\{\NF\left(g, G\setminus \{g\}\right)\right\} \cup
    G\setminus\{g\}
    $
    for all $g\in G$.
  \item Let $g\in G$ such that there is some $g'\in G$ and a small
    cofactor $c$ of $\LM(g')$ such that $\LT(g)=\LT(g')\cdot c$ and
    $\sig(g) = \sig(g')\cdot \lambda(c)$. $f$ is (weakly)
    $s$-reducible with respect to $G$ if and only if it is (weakly)
    $s$-reducible with respect to $G\setminus \{g\}$.
  \end{enumerate} 
\end{lem}
\begin{proof}
  \mbox{}
  \begin{enumerate}
  \item If $f$ is $s$-reducible (resp.\@ weakly $s$-reducible) with
    respect to $G\setminus \{g\}$ or if $\NF\left(g, G\setminus
      \{g\}\right) = g$, then there is nothing to show.

    Otherwise, there is a small cofactor $c$ of $\LM(g)$, such that
    $\LM(g)\cdot c = \LM(f)$ and $\sig(g)\cdot \lambda(c) < s$
    (resp.\@ $\sig(g)\cdot \lambda(c) \le s$), and there is some
    $g'\in G\setminus \{g\}$ and a small cofactor $c'$ of $\LM(g')$
    such that $\LM(g')\cdot c' = \LM(g)$ and $\sig(g')\cdot
    \lambda(c') < \sig(g)$.

    The monomial $c'\cdot c$ is a small cofactor of $\LM(g')$. Since
    $\sig(g')\cdot \lambda(c'\cdot c) = (\sig(g')\cdot
    \lambda(c'))\cdot \lambda(c) < \sig(g)\cdot \lambda(c)< s$
    (resp.\@ $...\le s$), we find that $f$ is (weakly) $s$-reducible
    with respect to $\{g'\}$ and thus with respect to $\tilde G$.
  \item If $f$ is (weakly) $s$-reducible with respect to $G\setminus \{g\}$
    then it is (weakly) $s$-reducible with respect to $G$.

    If $f$ is (weakly) $s$-reducible with respect to $\{h\}$ for some
    $h\in G\setminus \{g\}$ then $f$ is (weakly) $s$-reducible with
    respect to $G\setminus \{g\}$. Otherwise, if $f$ is (weakly)
    $s$-reducible with respect to $G$ then there is a small cofactor
    $d$ of $\LM(g)$ such that $\LM(g)\cdot d=\LM(f)$ and $\sig(g)\cdot
    \lambda(d) < s$ (or $...\le s$). It follows that $c\cdot d$ is a
    small cofactor of $\LM(g')$, and $\LM(g')\cdot (c\cdot d)=\LM(f)$
    and $\sig(g')\cdot \lambda(c\cdot d)<\sig(g)\cdot
    \lambda(d)<s$ (or $...\le s$). Hence, $f$ is (weakly)
    $s$-reducible with respect to $\{g'\}$ and thus with respect to
    $G\setminus \{g\}$.
  \end{enumerate}
\end{proof}

\begin{cor}\label{cor:sig_interred}
  For any finite signed subset $G$ of $M\setminus\{0\}$, there is an
  \emph{interreduced} finite signed subset $\interred(G)$ of
  $M\setminus\{0\}$ such that if $f$ is (weakly) $s$-reducible with
  respect to $G$ then $f$ is (weakly) $s$-reducible with respect to
  $\interred(G)$, for all $f\in F$ and $s\in \Mon(E)$.
\end{cor}
\begin{proof}
  If there is some $g\in G$ such that $g'=\NF(g,G\setminus \{g\})\not=g$ and
  $\LM(g')<\LM(g)$, then we replace $G$ by $\{g'\}\cup G\setminus \{g\}$. By
  the preceding Lemma, the change of $G$ does not decrease the set of elements
  of $F$ that are (weakly) $s$-reducible with respect to $G$, for all $s\in
  \Mon(E)$.

  If there is some $g,g'\in G$ and a small cofactor $c$ of $\LM(g')$ such that
  $\LM(g')\cdot c=\LM(g)$ and $\sig(g')\cdot \lambda(c)=\sig(g)$, then we
  replace $G$ by $G\setminus \{g\}$. By the preceding lemma, the set of
  elements of $F$ that are (weakly) $s$-reducible with respect to $G$ does not
  change, for all $s\in \Mon(E)$.

  Since the leading monomials strictly decrease and we only have finitely many
  monomials occurring in $G$, we obtain $\interred(G)$ after finitely many
  steps.
\end{proof}

\begin{lem}\label{lem:LMdistinct}
  Let $G$ be an interreduced finite signed subset of $M$.  If $g_1,g_2\in G$,
  $g_1\not=g_2$, and $c_1,c_2$ are small cofactors of $\LM(g_1),\LM(g_2)$ such
  that $g_i\cdot c_i$ is $\sig(g_i)\cdot \lambda(c_i)$-irreducible with
  respect to $G$, for $i=1,2$, then $\LM(g_1)\cdot
  c_1\not=\LM(g_2)\cdot c_2$.
\end{lem}
\begin{proof}
  To obtain a contradiction, we assume that $g_1,g_2,c_1,c_2$ satisfy the
  hypothesis of the lemma, but $\LM(g_1)\cdot c_1=\LM(g_2)\cdot c_2$.

  Without loss of generality, let $\deg(\LM(g_1))\le \deg(\LM(g_2))$. Since
  the monomial $\LM(g_1)\cdot c_1=\LM(g_2)\cdot c_2$ corresponds to a unique
  path in the quiver, we can write $c_1=c_1'\cdot c_2$ with a small cofactor
  $c_1'$ of $\LM(g_1)$ such that $\LM(g_2)=\LM(g_1)\cdot c_1'$. Since $G$ is
  interreduced, we have $\sig(g_2)\not= \sig(g_1)\cdot
  \lambda(c_1')$. Hence, $\sig(g_2)\cdot \lambda(c_2)\not= \sig(g_1)\cdot
  \lambda(c_1')\cdot \lambda(c_2) = \sig(g_1)\cdot \lambda(c_1)$. Since
  $\LM(g_1\cdot c_1)=\LM(g_1)\cdot c_1 = \LM(g_2)\cdot c_2 = \LM(g_2\cdot
  c_2)$, it follows that either $g_1\cdot c_1$ is
  $\sig(g_1)\cdot \lambda(c_1)$-reducible or $g_2\cdot c_2$ is $\sig(g_2)\cdot
  \lambda(c_2)$-reducible with respect to $G$, which is impossible by
  construction.
\end{proof}

\begin{defn}
  Let $G$ be a finite signed subset of $M\setminus \{0\}$ and $g\in
  G$.
  \begin{enumerate}
  \item If $c$ is the cofactor of a minimal toppling of $\LM(g)$, we
    call $(g,c)$ a \emph{critical pair of type T with cofactor $c$} of $G$.
  \item If $p=(g,c)$ is a critical pair of type T of $G$, we define
    the \emph{S-polynomial} of $p$ as $\mathcal S(p)=\left(\poly(g)\cdot c,
    \sigma(g)\cdot \lambda(c)\right)$.
  \item If $c$ is a small cofactor of $\LM(g)$ and there is some
    $g'\in G$ such that $\LM(g')=\LM(g)\cdot c$ and
    $\sig(g')<\sig(g)\cdot \lambda(c)$, then we call $(g,g')$ a
    \emph{critical pair of type S with cofactor $c$} of $G$.
  \item If $p=(g,g')$ is a critical pair of type S with cofactor $c$ of $G$,
    we define the \emph{S-polynomial} of $p$ as
    $$ \mathcal S(p) = \left(\poly(g)\cdot c - \frac{\LC(g')}{\LC(g)}g', \sig(g)\cdot \lambda(c)\right)$$
  \end{enumerate}
\end{defn}

When we do not name the type of a critical pair, it can be either
type. By construction, if $p$ is a critical pair of $G$, then the
S-polynomial $S(p)$ is a signed element. Note that for type S, the
\emph{unsigned} element $\poly(g')$ is reducible with respect to
$\{\poly(g)\}$, but the signed element $g'$ is \emph{not} reducible
with respect to $\{g\}$, because its signature is too small. This is
why we need to take into account critical pairs of type S. Indeed, in
the unsigned case, the critical pairs of type T would actually be
enough to construct standard bases; see Theorem~\ref{thm:criterion}
and~\cite{Green:Habil}.


\begin{defn}
  Let $G$ be a finite signed subset of $M\setminus \{0\}$.
  \begin{enumerate}
  \item A critical pair $(g,c)$ of $G$ of type T is \emph{normal}, if $g$ is
    irreducible with respect to $G$, and $\sig(g)\cdot \lambda(c)$ is not
    leading monomial of an element of $\ker(ev)$.
  \item A critical pair $(g,g')$ of type S with cofactor $c$ of $G$ is
    \emph{normal}, if both $g$ and $g'$ are irreducible with respect to $G$,
    and $\sig(g)\cdot \lambda(c)$ is not leading monomial of an element of
    $\ker(ev)$.
  \end{enumerate}
\end{defn}

\begin{defn}
  A finite signed subset $G$ of $M\setminus \{0\}$ satisfies the \emph{\FF
    criterion}, if for each normal critical pair $p$ of $G$, there is some
  $g\in G$ and a small cofactor $c$ of $\LM(g)$ such that $\sig(g)\cdot
  \lambda(c) = \sig\left(\mathcal S(p)\right)$ and $\poly(g)\cdot c$ is
  $\sig\left(\mathcal S(p)\right)$-irreducible with respect to $G$.
\end{defn}

\begin{lem}\label{lem:key}
  Let $G$ be a finite signed subset of $M\setminus \{0\}$ that
  satisfies the \FF criterion and is interreduced.
  Let $\tau$ be a monomial of $E$ that is not leading monomial of an
  element of $\ker(ev)$.

  Assume that there is some $g \in G$ and some monomial $b\in \mathcal
  B_>(\mathcal A)$ such that $\sig(g)\cdot \lambda(b)=\tau$. One can
  choose $g$ and $b$ such that $b$ is a small cofactor of $\LM(g)$ and
  $\poly(g)\cdot b$ is $\tau$-irreducible with respect to $G$.
\end{lem}
\begin{proof}
  The proof is by induction on $\deg (b)$.

  If there are $g$ and $b$ such that $\deg(b)=0$, then $\poly(g)\cdot
  b=\poly(g)$ and $\sig(g)\cdot \lambda(b)=\sig(g) =\tau$. Since $G$
  is interreduced, $\poly(g)$ is $\tau$-irreducible with respect to
  $G$.

  By now, let $\deg (b)>0$, so that we can write $b=b'\cdot x$ for
  some $x\in \mathcal B_>(\mathcal A)$ with $\deg(x)=1$. Let
  $\tau'=\sigma\cdot \lambda(b')$. Since $\tau=\tau'\cdot
  \lambda(x)$ is not leading monomial of an element of $\ker(ev)$,
  its divisor $\tau'$ isn't either. Thus, by induction, there is some
  $g_0\in G$ and a small cofactor $c_0'$ of $\LM(g_0)$ such that
  $\sig(g_0)\cdot \lambda(c_0')=\tau'$ and $\poly(g_0)\cdot c_0'$ is
  $\tau'$-irreducible with respect to $G$. Either $x$ is a toppling of
  $\LM(g_0)\cdot c_0'$, or it is a small cofactor of $\LM(g_0)\cdot
  c_0'$.

  Assume that $x$ is the cofactor of a toppling of $\LM(g_0)\cdot
  c_0'$. Since $x$ is of degree one, the toppling is minimal. Since
  $c_0'$ is a small cofactor of $\LM(g_0)$, it follows that
  $c_0=c_0'\cdot x$ is a minimal toppling of $\LM(g_0)$. Since $G$ is
  interreduced, $\poly(g_0)$ is $\sig(g_0)$-irreducible with respect
  to $G$. Moreover, $\sig(g_0)\cdot \lambda(c_0)=\tau$ is not the
  leading monomial of an element of $\ker(ev)$. Hence, $(g_0,c_0)$ is
  a normal critical pair of $G$ of type T.

  By the \FF criterion, there is some $g_1\in G$ and a small cofactor
  $c_1$ of $\LM(g_1)$ such that $\sig(g_1)\cdot
  \lambda(c_1)=\sig(\mathcal S(g_0,c_0))$ and $\poly(g_1)\cdot c_1$
  is $\sig(\mathcal S(g_0,c_0))$-irreducible with respect to
  $G$. Since $\sig(\mathcal S(g_0,c_0))=\sig(g_0)\cdot
  \lambda(c_0'\cdot x)=\tau$, the statement of the lemma holds in
  this case.

  There remains to study the case that $x$ is 
  a small cofactor of $\LM(g_0)\cdot c_0'$, which implies that $c_0 =
  c_0'\cdot x$ is a small cofactor of $\LM(g_0)$. We have
  $\sig(g_0)\cdot \lambda(c_0)=\tau$.

  Assume that $\poly(g_0)\cdot c_0$ is $\tau$-reducible with respect
  to $G$. That means, there is some $h\in G$ and a small cofactor $d$
  of $\LM(h)$ such that $\LM(h)\cdot d=\LM(g_0)\cdot c_0$ and
  $\sig(h)\cdot \lambda(d)<\tau$.
  The monomial $\LM(h)\cdot d=\LM(g_0)\cdot c_0=\LM(g_0)\cdot
  c_0'\cdot x$ corresponds to a unique path in our quiver $Q$. That
  path ends with the arrow that corresponds to $x$.

  Assume that $\deg (d)>0$. It follows that $d$ can be written as
  $d=d'\cdot x$ for some $d'\in \mathcal B_>(\mathcal A)$. Then,
  $\LM(h)\cdot d'=\LM(g_0)\cdot c_0'$. But $\poly(g_0)\cdot c_0'$ is
  $\tau'$-irreducible with respect to $G$. Hence, $\sig(h)\cdot
  \lambda(d')\ge \tau'$. It follows $\sig(h)\cdot \lambda(d)\ge
  \tau'\cdot \lambda(x)=\tau$. This is a contradiction and implies
  $\deg(d)=0$.

  Hence, we have some $h\in G$ such that $\LM(g_0)\cdot c_0=\LM(h)$
  and $\sig(h)<\tau$. Thus, $(g_0,h)$ is a critical pair of $G$ of
  type S with cofactor $c_0$. Both $g_0$ and $h$ are irreducible with
  respect to $G$, since $G$ is interreduced. Moreover, $\sig(g_0)\cdot
  \lambda(c_0)=\tau$ is not leading monomial of an element of
  $\ker(ev)$. Hence, $(g_0,h)$ is a \emph{normal} critical pair. We have
  $\sig(\mathcal S(g_0,h))=\tau$. Hence, by the \FF criterion, there
  is some $\tilde g\in G$ and a small cofactor $\tilde c$ of
  $\LM(\tilde g)$ such that $\sig(\tilde g)\cdot \lambda(\tilde
  c)=\tau$ and $\poly(\tilde g)\cdot \tilde c$ is $\tau$-irreducible
  with respect to $G$.
\end{proof}

Recall that we have provided our free module $E$ with generators
$\mathfrak e_1,...,\mathfrak e_m$ that are mapped by $ev$ to the
originally given generators $\hat g_1,...\hat g_m$ of $M$. We now come
to the main theorem of this section.
\begin{thm}\label{thm:sigBuchberger}
  Let $G$ be an interreduced finite signed subset of $M\setminus
  \{0\}$. Assume that for each $i=1,...,m$ so that $\mathfrak e_i$ is
  not leading monomial of an element of $\ker(ev)$, there is some
  $g\in G$ with $\sig(g)=\mathfrak e_i$. Then $G$ is a signed standard
  basis of $M$ if and only if $G$ satisfies the \FF criterion.
\end{thm}
\begin{proof}
  If $G$ is a signed standard basis, then any signed element of
  $M\setminus \{0\}$ is weakly reducible with respect to $G$. In
  particular, this holds for S-polynomials of critical pairs. Hence,
  $G$ satisfies the \FF criterion.

  Now suppose that $G$ satisfies the \FF criterion. We prove by
  contradiction that $G$ is a signed standard basis. Assume that there
  is a signed element $f$ of $M\setminus \{0\}$ so that $\poly(f)$ is
  $\sig(f)$-irreducible with respect to $M\setminus \{0\}$ and is not
  weakly $\sig(f)$-reducible with respect to $G$.

  Since $f$ is a signed element, there is some $\tilde f\in E$ with
  $\LM(\tilde f)=\sig(f)$ and $ev(\tilde f)=\poly(f)$. Assume that
  $\sig(f)$ is the leading monomial of an element $z\in\ker(ev)$. Then
  let $\tilde f'=\tilde f-\frac{\LC(\tilde f)}{\LC(z)}z$. Apparently
  $\LM(\tilde f')<\LM(\tilde f)$, but $ev(\tilde f')=ev(\tilde
  f)=\poly(f)$. Hence, $\poly(f)$ is $\sig(f)$-reducible with respect
  to $M\setminus \{0\}$. This contradiction implies that $\sig(f)$ is
  not the leading monomial of an element of $\ker(ev)$.

  By our assumption on the monomial ordering, any
  descending sequence of monomials of $E$ ends among the leading
  monomials of $\ker(ev)$ after finitely many steps. Hence, since
  $\sig(f)$ is not leading monomial of an element of $\ker(ev)$, we
  can choose $f$ such that $\sig(f)$ is minimal.

  We show: If $f'$ is a (not necessarily irreducible) signed
  element of $M\setminus \{0\}$ with $\sig(f')<\sig(f)$, then $f'$ is weakly
  reducible with respect to $G$. Namely, choose a signed element $f''$ of
  $M\setminus \{0\}$ with $\LM(f'')=\LM(f')$, so that $\sig(f'')$ is
  minimal. Then, $f''$ is irreducible with respect to $M \setminus \{0\}$, and
  $\sig(f'')\le\sig(f')<\sig(f)$. Hence, by the choice of $f$, $f''$ is weakly
  reducible with respect to $G$, and since $\LM(f')=\LM(f'')$ and $\sig(f')\ge
  \sig(f'')$, $f'$ is weakly reducible with respect to $G$ as well.

  Since $\sig(f)$ is not the leading monomial of an element of $\ker(ev)$, we
  can write it as $\sig(f)=\mathfrak e_i\cdot \lambda(c)$ for some monomial
  $c$ of $\mathcal A$. By the hypothesis of this theorem, there is some $g\in
  G$ such that $\sig(g)\cdot \lambda(c)=\sig(f)$. By Lemma~\ref{lem:key}, we
  can choose $g$ and $c$ such that $c$ is a small cofactor of $\LM(g)$ and
  $\poly(g)\cdot c$ is $\sig(f)$-irreducible with respect to $G$.

  Our aim is to show that $\LM(g)\cdot c=\LM(f)$, which means that
  $\poly(f)$ is weakly $\sig(f)$-reducible with respect to $G$ and
  hence finishes the proof.

  Assume that $\LM(g)\cdot c<\LM(f)$. Since $g$ is a signed element,
  there is some $\tilde g\in E$ such that $\LM(\tilde g)=\sig(g)$ and
  $ev(\tilde g)=\poly(g)$. Then, let $\tilde f'= \tilde
  f-\frac{\LC(\tilde f)}{\LC(\tilde g)}\tilde g\cdot \lambda(c)$. By
  assumption, we have $\LM(\tilde f')<\LM(\tilde f)$ and
  $\LM(ev(\tilde f'))=\LM(ev(\tilde f)=\LM(f)$. But that is a
  contradiction to $\poly(f)$ being $\sig(f)$-irreducible with respect
  to $M\setminus \{0\}$.

  Hence, $\LM(g)\cdot c\ge \LM(f)$. Assume that $\LM(g)\cdot c>\LM(f)$. Then,
  let $\tilde g' = \tilde g\cdot \lambda(c)-\frac{\LC(\tilde g)}{\LC(\tilde
    f)}\tilde f$. We have $\LM(\tilde g')<\LM(\tilde g\cdot
  \lambda(c))=\sig(f)$ and $\LM(ev(\tilde g'))=\LM(ev(\tilde g\cdot
  \lambda(c)))=\LM(g\cdot c)$.

  By the same argument as above, since $\sig(f)$ is minimal and since $\tilde
  g'$ yields a signed element of signature $\LM(\tilde g')<\sig(f)$, we obtain
  that $ev(\tilde g')$ is weakly $\LM(\tilde g')$-reducible with respect to
  $G$, and is thus $\sig(f)$-reducible with respect to $G$. But $\LM(ev(\tilde
  g'))=\LM(g\cdot c)$, and thus $g\cdot c$ is $\sig(f)$-reducible with respect
  to $G$. This is a contradiction to the choice of $g$ and $c$.

  To summarise, both $\LM(g)\cdot c<\LM(f)$ and $\LM(g)\cdot c>\LM(f)$
  yield a contradiction. Hence $\LM(g)\cdot c=\LM(f)$, and thus
  $f$ is weakly reducible with respect to $G$.
\end{proof}

\subsection{Computing signed standard bases}
\label{sec:sig_compute}

Theorem~\ref{thm:sigBuchberger} provides a way to compute signed
standard bases --- with the complication that one needed to know
$\ker(ev)$ in advance, in order to decide whether a critical pair is
normal. This is often unfeasible or impossible.
Therefore, we use a weakened version of the \FF criterion. The basic
idea is to use partial knowledge of the leading monomials of
$\ker(ev)$, and increase the partial knowledge on the fly. We have learnt this
idea from~\cite{ArriPerry:F5revised}.

In the rest of this section, let $L$ be a finite set of monomials of
$E$ such that each element of $L$ is leading monomial of some element
of $\ker(ev)$.

\begin{defn}
  A monomial
  $\mathfrak e_i\cdot c$ of $E$ is called \emph{standard relative to
    $L$}, if $c\in \stdmon_{\mathcal A}(\mathcal P)$ and there are no
  monomials $c',b,d$ of $\mathcal P$ with $\mathfrak e_i\cdot c'\in L$ and
  $c=b\cdot c'\cdot d$.
\end{defn}

\begin{rk}
  If $\mathfrak e_i\cdot m$ is not the leading monomial of an element
  of $\ker(ev)$ then it is standard relative to $L$, for any finite
  set $L$ of leading monomials of elements of $\ker(ev)$.
\end{rk}

\begin{defn}
  A critical pair $(g,c)$ of type T (resp.\@ a critical pair $(g,g')$ of type
  S) with cofactor $c$ of $G$ is called \emph{normal relative to $L$}, if
  $\sig(g)\cdot \lambda(c)$ is standard relative to $L$, and $g$ is
  irreducible (resp. both $g$ and $g'$ are irreducible) with respect to $G$.
\end{defn}
\begin{rk}
  If a critical pair of $G$ is normal, then it is normal relative to
  $L$, for any choice of $L$.
\end{rk}

\begin{algorithm}[H]
  \DontPrintSemicolon

  \KwData{A finite ordered subset $\{\hat
    g_1,...,\hat g_d\}$ of $M$, generating $M$ as a right-$\mathcal A$
    module.}

  \KwResult{An interreduced signed standard basis of $M$.}
  \Begin{ 
    $G\longleftarrow \interred\left(\{(\hat g_i, \mathfrak e_i)\co i=1,...,d\}\right)$\;
    $L = \emptyset$\;
    \While{ there is a critical pair $p$ of $G$ that is normal relative to $L$ so that it is impossible to find
      $g\in G$ and a small cofactor $c$ of $\LM(g)$ such that
      $\sig(g)\cdot \lambda(c)=\sig(\mathcal S(p))$ and $\poly(g)\cdot c$ is
      $\sig(\mathcal S(p)$-irreducible with respect to $G$}
    {
      $s\longleftarrow \NF(\mathcal S(p), G)$\;
      \If{$\poly(s)=0$}
      {
        $L\longleftarrow L\cup \sig(s)$\;
      }
      \Else
      {
        $G\longleftarrow \interred\left(G\cup \{s\}\right)$, enlarging $L$ if a
        zero reduction occurs\;
      }
    }
    \Return $G$\;
  }
  \caption{The \FF algorithm, computing signed standard bases}
  \label{alg:F5}
\end{algorithm}

\begin{thm}
  If Algorithm~\ref{alg:F5} terminates, then it returns an interreduced
  signed standard basis of $M$. It terminates in finite time, if and only if
  the while loop is executed only finitely many times.
\end{thm}
\begin{rk}
  The number of normal critical pairs of type T usually is \emph{much} smaller
  than the number of topplings considered in
  Algorithm~\ref{alg:Buchberger}. Hence, the \FF criterion discards many
  topplings whose S-polynomials would reduce to zero.

  One should note, however, that some additional critical pairs need to be
  considered, namely those of type S. However, in practical computations, we
  found that Algorithm~\ref{alg:F5} is a lot more efficient than
  Algorithm~\ref{alg:Buchberger}.
\end{rk}
\begin{proof}
  Let $p$ be as in the while-loop of Algorithm~\ref{alg:F5}. If
  $\NF(\mathcal S(p),G)=0$, then $\sig(\mathcal S(p))$ is the leading
  monomial of an element of $\ker(ev)$. Hence, in all steps of the
  algorithm, the set $L$ is formed by some leading monomials of
  elements of $\ker(ev)$. By the preceding remark, if $p$ is normal,
  then it is normal relative to $L$.

  By the initial definition of $G$, there is some $g\in G$ with
  $\sig(g)=\mathfrak e_i$, for any $\mathfrak e_i$. Of course, this still
  holds when adding an element to $G$ in the while-loop. We show: If there is
  some $g\in G$ with $\sig(g)=\mathfrak e_i$ then either $\mathfrak e_i$ is
  the leading monomial of an element of $\ker(ev)$, or $\interred(G)$ contains
  an element of signature $\mathfrak e_i$.
  Namely, when interreducing $G$,
  either $\poly\left(\NF(g,G)\right)=0$ and $g$ is removed from $G$, or $g$ is
  replaced by $\NF(g,G)$.  In the former case, $\sig(g)=\mathfrak e_i$ turns
  out to be the leading monomial of an element of $\ker(ev)$. In the latter
  case, it suffices to note that $\sig(g)=\sig\left(\NF(g,G)\right)=\mathfrak
  e_i$.
  Hence, $G$ satisfies the hypothesis of
  Theorem~\ref{thm:sigBuchberger}, which means that $G$ is a signed
  standard basis if and only if it satisfies the \FF criterion.

  If $G$ is not a signed standard basis, then the \FF criterion does
  not hold for $G$, and we will find a critical pair $p$ satisfying
  the hypothesis of the while-loop of Algorithm~\ref{alg:F5}, so that
  $G$ will be enlarged by $\NF(\mathcal S(p))\not=0$. In particular, the
  algorithm does not terminate yet.
  If $G$ is a signed standard basis, then the \FF criterion holds for
  $G$, and Algorithm~\ref{alg:F5} terminates, potentially after
  verifying that the S-polynomials of all remaining normal critical
  pairs relative to $L$ (which are finite in number) reduce to zero.

  Therefore, if Algorithm~\ref{alg:F5} terminates, then it returns an
  interreduced signed standard basis.  For the second statement, we note that
  each computation in the while-loop of Algorithm~\ref{alg:F5} is finite,
  since we use an admissible monomial ordering.
\end{proof}

To test the condition of the while-loop of Algorithm~\ref{alg:F5} is
certainly computationally complex. However, if $p$ is a critical pair,
it helps that $\sig(\mathcal S(p))$ can be read off of $p$ without
computing $\mathcal S(p)$. In practical implementations, one computes
a list of normal critical pairs and updates that list whenever $L$ or
$G$ change. We do not go into detail, but remark that the following
lemma helps to keep the list short.
\begin{lem}
  If $p$ and $p'$ are two critical pairs of $G$ such that $\sig\left(\mathcal
  S(p)\right) = \sig\left(\mathcal S(p')\right)$ and $p$ is considered in the
  while-loop of Algorithm~\ref{alg:F5}, then $p'$ will not be considered in
  the while-loop.  In particular, each critical pair will be considered at
  most once during the algorithm.
\end{lem}
\begin{rk}
  The fact expressed in this lemma is usually referred to as the
  \emph{rewritten criterion}~\cite{Faugere:F5}.
\end{rk}
\begin{proof}
  When $p$ is considered in the while-loop, two possibilities occur, depending
  on whether $\poly(\NF(\mathcal S(p),G))=0$ or not.

  If $\poly(\NF(\mathcal S(p),G))=0$ then $\sig(\mathcal S(p))$ will
  be added to $L$, and thus $p'$ will be discarded since it is not
  normal relative to the enlarged set $L$.

  If $\poly(\NF(\mathcal S(p),G))\not=0$, then we add to $G$ the signed
  element $\NF(\mathcal S(p),G)$, that is irreducible with respect to $G$ and
  has the same signature as $\mathcal S(p)$ and thus the same signature as
  $\mathcal S(p')$. Hence, $p'$ will be discarded in this case as well, by the
  \FF criterion.
\end{proof}

\section{Loewy layers of right modules over basic algebras}
\label{sec:Loewy}

In this section, let $\mathcal P$ be a path algebra over a field $K$, and let
$\mathcal A$ be a \emph{basic algebra}. Hence, $\mathcal A$ is
finite-dimensional over $K$, and if $\psi\co \mathcal P\to \mathcal A$ is the
quotient map, then $\ker(\psi)\subset \mathcal P^2$. In particular, any choice
of a monomial ordering on $\mathcal P$ induces an admissible monomial ordering
on $\mathcal A$.

Recall that a generating set of $\mathcal P$ as a $K$-algebra is given
by idempotents $1_v$ corresponding to the vertices of a quiver $Q$,
and elements $x_i$ of degree one corresponding to the arrows of
$Q$. Since the defining relations for $\mathcal A$ are at least quadratic,
$$\Rad(\mathcal A)=\{\psi(x)\co x\in\mathcal P, \deg(x)=1\}\cdot \mathcal A$$
and $\ker(\psi)$ does not contain elements of degree zero or one.

Apart from the additional assumption on the relations, we use the same
notations as in the preceding sections. Hence, we have a sub-module $M$ of a
free right-$\mathcal A$ module $F$ of rank $r$, generated by $\{\hat
g_1,...,\hat g_m\}$, and a free right-$\mathcal A$ module $E$ of rank $m$
whose generators $\mathfrak e_1,...,\mathfrak e_m$ are mapped to $\hat
g_1,...,\hat g_m$ by a homomorphism $ev\co E\to M$.  We have $\Rad^k(M) =
M\cdot \left(\Rad(\mathcal A)\right)^k$ for all $k=1,2,...$.

\begin{lem}
  Under the hypotheses of this section, Algorithm~\ref{alg:F5}
  terminates.
\end{lem}
\begin{proof}
  In each repetition of the while-loop in Algorithm~\ref{alg:F5},
  there is only a finite number of critical pairs to be considered,
  simply since $G$ is finite and the number of minimal topplings of
  each monomial is finite.

  If $p$ is the critical pair considered in the while-loop, and
  $\poly\left(\NF(p,G)\right)=0$, then $G$ does not change. In particular, no
  additional critical pair emerges. On the contrary, some normal critical
  pairs may become non-normal relative relative to the now enlarged $L$.

  Otherwise, the set
  $$
  \{\sig(g)\cdot \lambda(m)\co g\in G, m\text{ a small cofactor of }\LM(g)\}
  $$
  strictly increases. Since $E$ only has finitely many monomials that
  are not leading monomials of elements of $\ker(ev)$, this can only
  happen finitely many times.
\end{proof}

\begin{defn}
  Let $\tau$ be a monomial of $E$. We define the
  \emph{$\tau$-layer} of $M$ as $$\mathcal L_\tau(M) =
  \{ev(\tilde f)\co \tilde f\in E, \LM(\tilde f)\le \tau\}.$$
\end{defn}

\begin{lem}\label{lem:Loewy_tau}
  Assume that the monomial ordering on $\mathcal E$ is a negative
  degree ordering. Let $\tau$ be the greatest monomial of $E$ such
  that $\deg(\tau)=d$, for some non-negative integer $d$. Then,
  $\Rad^d(M)=\mathcal L_\tau(M)$.
\end{lem}
\begin{proof}
  Since $\Rad^d(M) = M\cdot \left(\Rad(\mathcal A)\right)^d$ and
  $\Rad(\mathcal A)$ is generated by the monomials of $\mathcal A$ of
  degree one, it follows that $f\in M$ belongs to $\Rad^d(M)$ if and
  only if there is some $\tilde f\in E$ whose monomials are all of
  degree at least $d$, and $ev(\tilde f)=f$.

  Since we assume that the monomial ordering on $E$ is a negative
  degree ordering, the monomials of $\tilde f$ are all of degree at
  least $d$ if and only if $\LM(\tilde f)\le \tau$.
\end{proof}

\begin{lem}\label{lem:basis_Ltau}
  Let $\tau$ be a monomial of $E$. Let $G$ be an interreduced signed standard
  basis of $M$.
  Let $B_\tau(M,G)$ be the set of all $\poly(g)\cdot c$ with $g\in G$
  and a small cofactor $c$ of $\LM(g)$ such that $\sig(g)\cdot
  \lambda(c)\le \tau$ and $\poly(g)\cdot c$ is $\sig(g)\cdot
  \lambda(c)$-irreducible with respect to $G$.
  Then, $B_\tau(M,G)$ is a $K$-vector space basis of $\mathcal L_{\tau}(M)$.
\end{lem}
\begin{proof}
  If $f\in \mathcal L_\tau(M)$, then there is some monomial $\sigma\le\tau$ of
  $E$ such that $(f,\sigma)$ is a signed element of $M$. Since $G$ is a signed
  standard basis of $M$, $f$ is weakly $\sigma$-reducible with respect to
  $G$. Hence, we find $g\in G$ and a small cofactor $c$ of $\LM(g)$ such that
  $\sig(g)\cdot \lambda(c)\le \sigma\le \tau$ and $\LM(g)\cdot c=\LM(f)$.
  When we choose $\sig(g)\cdot \lambda(c)$ minimal, then $\poly(g)\cdot c$ is
  $\sig(g)\cdot \lambda(c)$-irreducible with respect to $G$.
  Hence, $B_d(M,G)$ generates $L_\tau(M)$ as a $K$-vector space.

  The leading monomials of the elements of $B_\tau(M,G)$ are pairwise
  distinct, by Lemma~\ref{lem:LMdistinct}, since $G$ is interreduced. Hence,
  $B_\tau(M,G)$ is $K$-linearly independent.
\end{proof}

\begin{rk}
  If $\tau,\tau'$ are monomials of $E$ and $\tau'\le \tau$, then
  $B_{\tau'}(M,G)\subseteq B_\tau(M,G)$, for any interreduced signed standard
  basis $G$ of $M$.
\end{rk}

Recall that the $d$-th Loewy layer of $M$ is $\Rad^{d-1}(M)/\Rad^d(M)$, for
$d=1,2,...$. For $f\in \Rad^{d-1}(M)$, we denote the equivalence class
of $f$ in $\Rad^{d-1}(M)/\Rad^d(M)$ by $[f]$.

\begin{thm}\label{thm:layer}
  Suppose that the monomial ordering on $E$ is a negative degree ordering. Let
  $d$ be some positive integer, let $\tau$ be the greatest monomial of $E$
  such that $\deg(\tau)=d-1$, and let $\tau'$ be the greatest monomial of $E$
  such that $\deg(\tau')=d$. Let $G$ be an interreduced signed standard basis
  of $M$.  Then 
  $$\left\{[f]\co f\in B_\tau(M,G)\setminus B_{\tau'}(M,G)\right\}$$
  is a $K$-vector space basis of the $d$-th Loewy layer of $M$.
\end{thm}
\begin{proof}
  By the choice of $\tau$ and $\tau'$ and by Lemmas~\ref{lem:Loewy_tau}
  and~\ref{lem:basis_Ltau}, $B_\tau(M,G)$ is a basis of $\Rad^{d-1}(M)$, and
  its subset $B_{\tau'}(M,G)$ is a basis of $\Rad^{d}(M)$. The claim directly
  follows.
\end{proof}

\begin{cor}
  If $G$ is an interreduced signed standard basis of $M$. The elements of $G$
  whose signatures are of degree zero form a minimal generating set of $M$.
\end{cor}
\begin{proof}
  By Theorem~\ref{thm:layer}, the elements of $G$ with signatures of degree
  zero yield a basis of the first Loewy layer of $M$, \emph{i.e.}, of the head
  of $M$. Hence, they form a minimal generating set of $M$.
\end{proof}

\bibliographystyle{alpha}
\bibliography{../references}

\end{document}